\documentclass{amsart}

\usepackage{amssymb,amsmath,amsthm}

\newtheorem{theorem}{Theorem}
\newtheorem{lemma}[theorem]{Lemma}
\newtheorem{conjecture}[theorem]{Conjecture}
\newtheorem{proposition}[theorem]{Proposition}
\newtheorem{corollary}[theorem]{Corollary}

\newcommand{\binomial}[2]{{#1 \choose #2}}
\newcommand{\gaussian}[2]{{#1 \brack #2}}

\newcommand{\bound}[1]{\mathsf{s}_{\le #1}}
\newcommand{\Dav}{\mathsf{D}}
\newcommand{\N}{\mathbb N}

\begin{document}

\title[Coding theory and Davenport constants]
{An application of coding theory\\ to estimating Davenport constants}

\author[A. Plagne, W.~A. Schmid]{Alain Plagne and Wolfgang A.\ Schmid}

\address{\rm \noindent Alain {\sc Plagne}, Wolfgang A. {\sc Schmid}\newline\indent
Centre de Math{\'e}matiques Laurent Schwartz\newline\indent
UMR 7640 du CNRS\newline\indent
{\'E}cole polytechnique\newline\indent
91128 Palaiseau cedex\newline\indent
France}
\email{plagne@math.polytechnique.fr, schmid@math.polytechnique.fr}

\thanks{
W.S. is supported by the Austrian Science Fund (FWF): J 2907-N18}

\begin{abstract}
We investigate a certain well-established
generalization of the Davenport constant.
For $j$ a positive integer (the case $j=1$, is
the classical one) and a finite Abelian group $(G,+,0)$,
the invariant $\Dav_j(G)$ is defined as the smallest $\ell$
such that each sequence over $G$ of length at least $\ell$ has
$j$ disjoint non-empty zero-sum subsequences.
We investigate these quantities for elementary $2$-groups
of large rank (relative to $j$). Using tools from coding theory,
we give fairly precise estimates for these quantities.
We use our results to give improved bounds for
the classical Davenport constant of certain groups.
\end{abstract}

\maketitle

\section{Introduction}

For a given Abelian group $(G,+,0)$, the \emph{Davenport constant}, denoted  $\Dav (G)$,
is defined as the smallest integer $\ell$ such that
each sequence over $G$ of length at least $\ell$ has a non-empty zero-sum subsequence,
\emph{i.e.} the sum of its terms is $0$.
Equivalently, $\Dav (G)$
is the maximal length of a minimal zero-sum sequence over $G$, \emph{i.e.}
the maximal length of a sequence of elements of $G$ summing to
$0$ and with no proper subsequence summing to $0$.

It is considered as a central object in combinatorial number theory
since Davenport popularized it in the 60's (as reported in \cite{Ol}),
notably for its link with algebraic number theory, see \emph{e.g.}
\cite{GH} or \cite{GB}.
In fact it seems that the first paper that deals with this invariant
was written by Rogers \cite{Ro}, who himself attributes the paternity of
the problem to Sudler.

This invariant has become the prototype of algebraic invariants
of combinatorial flavour. Since the 60's, the theory of these invariants
has highly developed in several directions; see for instance the survey article
\cite{GG} or Chapters 5, 6, and 7 of \cite{GH}.

Let $G$ be written, as is always possible, as a direct sum of cyclic
groups $G\cong C_{n_1} \oplus \cdots \oplus C_{n_r}$ with integers
$1< n_1 \mid \dots \mid n_r$ ($r$ denotes the rank of $G$,
and $n_r$ the exponent, except for $r=0$ where
the exponent is $1$). Then, the basic lower bound for Davenport constant is
\begin{equation}
\label{lowb}
\Dav (G) \geq 1+ \sum_{i=1}^r (n_i - 1);
\end{equation}
to see this, note that a sequence containing only, for each cyclic
component $C_{n_i}$ ($1 \leq i \leq n$), one generating element $n_i-1$ times,
has no non-empty zero-sum subsequence.

It is known that for groups of rank at most two
and for $p$-groups ($p$, a prime),
inequality \eqref{lowb} is in fact an equality;
this was obtained independently in \cite{vEB} and \cite{Ol,Ol2}.
For groups of rank at least four, equality is definitely not
the rule (see \cite{Baa, vEB,GS}).
In the case of groups of rank three, it is conjectured that equality
always holds but this conjecture is wide open (see \cite{GG}).

Concerning upper bounds, the best general result is the following
\[
\Dav (G) \leq \exp(G) \left( 1 + \log \frac{|G|}{\exp(G)} \right)
\]
proved in \cite{vEBK, Meshu}.

In view of the depicted situation, it appears that it is generically
very difficult to determine the Davenport constant of a group
(of rank at least three).
In particular, despite various works related to the Davenport constant over the years,
its actual value was only determined for a few additional -- beyond the ones known
since the end of the 60's -- families of groups; see \cite{BS} for a recent contribution.

The \emph{$j$-wise Davenport constants} are defined depending
on a positive integer $j$. We define $\Dav_j(G)$ to be the smallest
$\ell$ such that each sequence over $G$ of length at least $\ell$
has $j$ disjoint non-empty zero-sum subsequences.
Equivalently, the maximal length of a zero-sum sequence over $G$
that cannot be decomposed into $j+1$ non-empty zero-sum sequences.
Evidently $\Dav_1 (G) = \Dav (G)$ and for any positive $j$ one has
$\Dav_j (G) \leq j \Dav (G)$.

This variant of the Davenport constant was introduced by Halter-Koch \cite{HK}
in order to determine the order of magnitude
of the counting function of the set of algebraic integers of some number field
that are not divisible by a product of $j+1$ irreducible algebraic
integers of this number field. It is
\[
\frac{x}{\log x} (\log \log x)^{\Dav_j(G)-1}\]
where $G$ denotes the ideal class group of the number field considered.
Moreover, knowledge of these invariants is useful
in investigations of the Davenport constant itself (cf.~below for
details and recall \emph{e.g.} that a considerable part of \cite{BS},
determining $\Dav(C_3^2\oplus C_{3n})$,
is devoted to determining $\Dav_j(C_3^3)$ and closely related problems).

The case of cyclic groups is the simplest one since then it is easy
to see that $jn-1 < \Dav_j (C_n)\le j \Dav (C_n)=jn$ (for the lower bound,
simply take a generating element repeated that number of times). The case of groups of rank two
is also known \cite{HK}, as well as the case of certain closely
related groups \cite[Section 6.1]{GH}.
But in general computing (even bounding)
$\Dav_j (G)$ is quite more complicated than for $\Dav (G)$,
in particular for (elementary) $p$-groups.
For example, $\Dav_j(G)$ for all $j$, is only known for the following
elementary $p$-groups of rank greater than two: $C_2^3$, $C_2^4$, $C_2^5$, and $C_3^3$ (see \cite{DOQ,FS,BS}).

The main difference is that it seems that the types of arguments used to
determine $\Dav(G)$ for $p$-groups cannot be applied.
At first this might seem surprising, however
recall that the same phenomenon is encountered for the invariant $\eta(G)$ (cf.~Section \ref{sec2})
and other closely related invariants such as the Erd\H{o}s--Ginzburg--Ziv constant.
In particular, the difficulty
of the problem seems to increase with the rank of the groups considered
(to be precise, the key quantity is the size of the rank
relative to the exponent), as then
$\Dav(G)$ and $\eta(G)$ are far apart (cf.~Section \ref{sec2} for the relevance of this
fact).

In this paper and from now on, we focus on the case
of elementary $2$-groups with large rank.
The reason for this is two-fold.
On the one hand, it is an interesting case;
the rank is `maximal' relative to the exponent.
On the other hand, the special nature of the group
allows certain arguments that fail in more general
situations; for instance recall that to show
$\eta(C_2^r)=2^r$ is almost trivial, yet the problem
of determining $\eta(C_p^r)$ for any (fixed) odd prime $p$
and arbitrary $r$ is wide open (it is even open, which
order of magnitude is to be expected).

While it is easy to determine that $\Dav_1(C_2^r)=r+1$ -- we effectively
consider a vector space over a field with a unique non-zero element, and having
no non-empty zero-sum subsequences is thus equivalent to linear
independence -- the situation becomes more complicated for larger values of $j$.
We observe that
\[
r< \Dav_1(C_2^r)\le \Dav_j(C_2^r) \le j \Dav_1(C_2^r) \le j(r+1).
\]
Thus, for a fixed positive integer $j$, the sequence
$(\Dav_j (C_2^r))_{r \in \N}$ has to grow linearly.
Our main aim here is to make this statement more precise and to study
the following relevant quantity
\[
\frac{\Dav_j (C_2^r)}{r},
\]
of which we examine the asymptotic behaviour when $r$ becomes large.
The quantities
\[
\alpha_j=\liminf_{r \rightarrow  +\infty} \frac{\Dav_j (C_2^r)}{r} \hspace{1cm}
\text{ and } \hspace{1cm}
\beta_j=\limsup_{r \rightarrow  +\infty} \frac{\Dav_j (C_2^r)}{r}
\]
will be considered and investigated. By the above crude reasoning
we only get that $1 \leq \alpha_j \leq \beta_j \le j$. Our first theorem,
which gives explicit estimates for small values of $j$, improves on this.
It is a generalization of results by Koml\'os (lower bound, quoted in
\cite{M, CL}) and by Katona and Srivastava (upper bound) \cite{KS}
for the case $j=2$ that we recall for the sake of completeness
in the statement of the theorem.

\begin{theorem}
\label{thm_expl}
For each sufficiently large integer $r$ we have
\[
\begin{array}{ccccc}
1.261 \ r &  \le &  \Dav_2(C_2^r) & \le &   1.396 \  r, \nonumber \\
1.500 \ r & \le & \Dav_3(C_2^r) & \le   &  1.771 \  r, \nonumber \\
1.723 \ r & \le & \Dav_4(C_2^r)  & \le  & 2.131 \  r, \nonumber \\
1.934 \ r & \le & \Dav_5(C_2^r)  & \le & 2.478 \  r, \nonumber \\
2.137 \ r & \le & \Dav_6(C_2^r)  & \le & 2.815 \  r, \nonumber \\
2.333 \ r & \le & \Dav_7(C_2^r)  & \le & 3.143 \  r, \nonumber \\
2.523 \ r & \le & \Dav_8(C_2^r)  & \le & 3.464 \  r, \nonumber \\
2.709 \ r & \le & \Dav_9(C_2^r)    & \le & 3.778 \  r, \nonumber \\
2.890 \ r & \le & \Dav_{10}(C_2^r) & \le & 4.087 \  r. \nonumber \\
\end{array}
\]
\end{theorem}

As will be apparent from our arguments, having our method at hand,
to expand this list further is merely a computational effort.

Notice that the question whether $\alpha_j=\beta_j$ (independently of the
value that this constant would take) seems not obvious.

After the study of small values for $j$, we turn our attention
to the case of large $j$'s.
Although we are unable to prove that $\alpha_j$ and $\beta_j$ are equal,
which seems conceivable, we show that they at least grow at the same speed
and more precisely determine their order of magnitude.

\begin{theorem}
\label{thm_asymp}
When $j$ tends to infinity, we have the following:
\[
\log 2 \
\left( \frac{j}{\log j} \right) \lesssim
\liminf_{r \rightarrow + \infty} \frac{\Dav_j (C_2^r)}{r} \le
\limsup_{r \rightarrow + \infty} \frac{\Dav_j (C_2^r)}{r} \lesssim
2 \log 2 \ \left( \frac{j}{\log j} \right) .
\]
\end{theorem}

We believe that the bound for the $\liminf$ is closer to the actual value.
A heuristic suggests that the $\limsup$ in Theorem \ref{thm_asymp}
is in fact close to  $\log 2\ (j /\log j)$ as well.
More precisely, we formulate the following conjecture.

\begin{conjecture}
\label{conjec}
For any positive integer $j$, the limit
\[
\gamma_j = \lim_{r \rightarrow + \infty} \frac{\Dav_j (C_2^r)}{r}
\]
exists and one has
\[
\gamma_j \sim \log 2\ \left( \frac{j}{\log j} \right)
\]
as $j$ tends to infinity.
\end{conjecture}

For results in the converse scenario, that is fixed but arbitrary $r$, and
$j$ goes to infinity, see \cite{FS}.

One of the main reasons for studying $j$-wise Davenport constants is
the fact that they are important in obtaining results on the Davenport
constant itself.
The connection is encoded in the following inequality, due to Delorme,
Ordaz, and Quiroz \cite{DOQ}.
For a finite Abelian group $G$ and a subgroup $H$ one has
\begin{equation}
\label{eq_DOQ}
\Dav(G) \le \Dav_{\Dav(H)}(G/H).
\end{equation}
Among others, this inequality encodes the classical form of the
inductive method, originally introduced to
determine the Davenport constant for groups of rank two (see \cite{Ro,vEB,Ol2}).

We can apply our results on the $j$-wise Davenport constants to obtain improved bounds
on the Davenport constant for certain types of groups (the $2$-rank has to be
`large' relative to the order).
We only formulate it explicitly for a quite special type of group,
which however, due to its extremal nature, is of relevance in this context.

\begin{corollary}
\label{corrr}
When $n$ tends to infinity, we have
\[
\limsup_{r \to +\infty} \frac{\Dav (C_2^{r-1} \oplus C_{2n})}{r}
\lesssim   2 \log 2\ \frac{n}{\log n}.
\]
\end{corollary}

For comparison, the general bound mentioned above yields only
\[
\limsup_{r \to +\infty} \frac{\Dav (C_2^{r-1} \oplus C_{2n})}{r}
\lesssim 2 \log 2\ n.
\]

We immediately give the short proof of this result.
\begin{proof}
By \eqref{eq_DOQ}, we get that
\[
\Dav(C_2^{r-1}\oplus C_{2n})\le \Dav_{\Dav(C_n)}(C_2^r)=\Dav_{n}(C_2^r).
\]
By Theorem \ref{thm_asymp}, the claim follows.
\end{proof}

We finish this Introduction with outlining the plan of the present article.
In Section \ref{sec2}, we explain the methods and the prerequisites
we need in the course of this article.
In Section \ref{sec4}, we derive the lower bounds of our two results
while in Section \ref{sec3} the upper bounds are proved.
Finally, in Section \ref{sec_heu}, we discuss the heuristic
leading to Conjecture \ref{conjec}. If true, this heuristic would establish
at least the asymptotic equivalence of $\alpha_j$ and $\beta_j$
and imply the second part of Conjecture \ref{conjec}.

\section{The methods}
\label{sec2}

We outline the methods we use to establish our results.

\subsection{Zero-sum subsequences of bounded length}
Let $G$ be an Abelian group.
For $x$ a real number, let $\bound{x}(G)$ denote the smallest element
$\ell\in \mathbb{N}\cup\{+\infty\}$ such that each sequence of length
at least $\ell$ has a zero-sum subsequence of length at most $x$.
Evidently, $\bound{x}(G) = \bound{\lfloor x \rfloor}(G)$,
yet for technical reasons it is useful to define $\bound{x}(G)$ for non-integral $x$ as well.

A prominent special case of this definition is $x=\exp(G)$, the resulting
invariant is typically denoted by $\eta(G)$;
note that for $x < \exp(G)$, one has $\bound{x}(G)=+\infty$.
Also, note that for $x\ge  \Dav(G)$ we have $\bound{x}(G)=\Dav(G)$.

For results on $\bound{x}(G)$, for generic $x$, mainly for elementary $2$- and $3$-groups,
see \emph{e.g.} \cite{DOQ,BS}, and \cite{CZ2} for a very closely related problem (cf.~below).
For recent results on $\eta(G)$, focusing on lower bounds, see \cite{EEGKR,Edel}.

To determine $\eta(G)$ seems to be
a very difficult problem in general. For example, $\eta(C_3^6)$ was determined only recently \cite{Pot},
despite the fact that the problem of determining $\eta(C_3^r)$ is fairly popular
(see \cite{EEGKR} for a detailed outline of several
problems, and their respective history, that are equivalent
to determining $\eta(C_3^r)$).

Yet, in the case of elementary $2$-groups, to determine both
$\eta(C_2^r)$ and $\Dav(C_2^r)$ is almost trivial.
However, for other values of $x$ even for elementary $2$-groups
the problem of determining $\bound{x}(C_2^r)$ is not at all trivial, namely, it is equivalent to a central
problem of coding theory (cf.~below).

It is known, in particular by the work of
Delorme, Ordaz, and Quiroz \cite{DOQ},
that the invariants $\bound{x}(G)$ can
be used to derive upper bounds for $\Dav_j(G)$.
More specifically, we have (this is Lemma 2.4 in \cite{FS})
\begin{equation}
\label{eq_rec}
\Dav_{j+1} (G) \le \min_{i\in \mathbb{N}} \max
\{ \Dav_j(G) + i , \bound{i}(G) - 1 \}.
\end{equation}
Thus, knowing $\Dav_1(G)=\Dav(G)$ and the constants $\bound{i}(G)$,
or bounds for these constants, one can obtain, recursively applying
estimate \eqref{eq_rec}, bounds for $\Dav_j(G)$.
Notice however that even exact knowledge of $\Dav(G)$ and $\bound{i}(G)$
for all $i$ can be insufficient to determine $\Dav_j(G)$ exactly
via this method, which in general is not optimal.

\subsection{Coding theory enters the picture}

We recall that the link between coding theory and
combinatorial number theory is not new. For instance, Cohen, Litsyn, and Z\'emor \cite{CLZ}
 used coding theoretic bounds in
the Sidon problem. The general paper \cite{CZ2} by two
of these authors provides a worthwhile introduction to the links
between the two problematics. Also, Freeze \cite{F} used
coding theory in the present context.

One of the reasons for this connection is the following folkloric lemma \cite{MWS}.

\begin{lemma}
\label{PCM}
The minimal distance of a binary linear code ${\mathcal C}$ is equal to
the minimal length of a zero-sum subsequence of columns of a
parity check matrix of ${\mathcal C}$.
\end{lemma}

In the case of present interest, Cohen and Z\'emor \cite{CZ2} pointed
out a connection between $\bound{i}(C_2^r)$ and coding theory.
As the situation at hand is slightly different from the one in that
paper, and this connection is central to our investigations, we recall
and slightly expand it in some detail.

First, we give a technically useful definition.
In the present context, we call a function
$f: [0,1] \rightarrow [0,1]$ \emph{upper-bounding}
if it is decreasing (not necessarily strictly), continuous and
each $[n,k,d]$ code satisfies
\[
\frac{k}{n} \le f \left( \frac{d}{n} \right).
\]
In other words, upper-bounding functions are the functions intervening in the upper bounds
of the rate of a code by a function of its normalized minimal distance.

We call a function \emph{asymptotically upper-bounding} if it
has the same properties as an upper-bounding function,
except that the inequality only has to hold for $[n,k,d]$ codes with sufficiently large $n$.

Notice that in both cases, the assumptions on decreasingness and continuity
are not restrictive at all since these assumptions are usually fulfilled.

\begin{lemma}
\label{upperbounding}
Let $f$  
be an upper-bounding function.
Let $d$, $n$, and $r$ be three positive integers satisfying
$2 \leq d \leq n-1 $ and
\[
\frac{n-r}{n} > f \left( \frac{d+1}{n} \right),
\]
then
\[
\bound{d}(C_2^r) \leq n.
\]
Moreover, the same assertion holds true for $f$ an asymptotically upper-bounding
function if we impose that $n$ is sufficiently large (depending on $f$).
\end{lemma}

\begin{proof}
Let $S=g_1, \dots, g_n$ be an arbitrary finite sequence over $C_2^r$. We shall
prove that it contains a zero-sum subsequence of length at most $d$.

It is immediate that $S$ has a zero-sum subsequence
of length $1$ if and only if $0$ occurs in $S$, and
that $S$ has a zero-sum subsequence of length $2$ if and only
if some element occurs at least twice in $S$.
Thus, since $d \ge 2$, we may assume that
$S$ does neither contain $0$ nor an element more than once \emph{i.e.}
we effectively have to study the case of subsets of $C_2^r\setminus \{0\}$
(and not the one of general sequences in $C_2^r$).

We assert that we may assume that the elements appearing in $S$ generate
$C_2^r$.
To see this, note that if $g$ is an element of $C_2^r$ not contained in
the subgroup generated by the elements of $S$, and if $T$ denotes the sequence
obtained by appending $g$ to $S$,
then each zero-sum subsequence of $T$ is in fact a zero-sum subsequence
of $S$. Thus, if $S$ has no zero-sum subsequence of length at most $d$
and the elements of $S$ do not generate $C_2^r$,
then the longer sequence $T$, defined as above,
neither has a zero-sum subsequence of length at most $d$.
So, it suffices to establish an upper bound on the length of
sequences $S$ such that the elements appearing in $S$ generate
$C_2^r$.

We choose some basis of $C_2^n$. We consider the binary linear code
${\mathcal C} \subset C_2^n$
of length $n$ whose parity check matrix is
$A=[g_1 \mid \dots \mid g_n] \in {\mathcal M}_{r,n}$
(identify the $g_i$'s with their coordinate vectors
with respect to some basis of $C_2^r$, and consider them as column vectors,
and use the just chosen basis of $C_2^n$).
Notice that the rank of this matrix is equal to $r$ in view of our assumption
that the $g_i$'s generate $C_2^r$. Let $m$ be the minimal distance of
${\mathcal C}$. By definition, the code ${\mathcal C}$
is an $[n,n-r,m]$ binary linear code.
But by assumption since $f$ is upper-bounding and
\[
\frac{n-r}{n} > f \left( \frac{d+1}{n} \right),
\]
an $[n,n-r,d+1]$ code cannot exist. This implies that $m < d+1$, or
equivalently $d \geq m$.

We conclude by applying Lemma \ref{PCM} which shows
that $S$  possesses a zero-sum subsequence of length $m$.

The additional claim for asymptotically upper-bounding functions,
is immediate in view of the just given argument.
\end{proof}

\section{Lower bounds}
\label{sec4}

In this section, we establish the lower bounds for
$\Dav_j(C_2^r)$, for large $r$, contained in
Theorems \ref{thm_expl} and \ref{thm_asymp}.
Specifically, we prove the following asymptotic lower bound in $r$,
which immediately yields both lower bounds (replacing
$\log (j+1)$ by $\log j$ is asymptotically, in $j$, irrelevant).

\begin{proposition}
\label{prop6}
Let $j$ be a positive integer. Then
\[
\Dav_{j}(C_2^r)\geq \log 2 \frac{j}{\log (j+1)}\ r
\]
as $r$ tends to infinity.
\end{proposition}

Notice that the case $j=1$ is essentially trivial,
while the case $j=2$ of this result, formulated in the context
of coding theory, is attributed to  Koml\'os in \cite{CL} and \cite{M}.
Indeed, our proof will generalize Koml\'os' approach (in the form given
in \cite{CL}). This proof can be seen as probabilistic, yet
we prefer to present it via a direct counting argument.
Notice that it is non-constructive.

We need the following well-known lemma (see \emph{e.g.} \cite{MWS}).

\begin{lemma}
\label{coeff}
Let $n$ and $k$ be two positive integers, $n \geq k$.
In an $n$-dimensional vector space over a field with $2$
elements, the number of $k$-dimensional subspaces
is equal to the $2$-ary binomial coefficient defined as
\[
\gaussian{n}{k} =
\frac{(2^n-1)\cdots (2^{n-k+1} -1)}{(2^k -1) \cdots (2-1)}.
\]
Moreover, the number of $k$-dimensional subspaces containing
a fixed $j$-dimensional subspace, $k \geq j$, is equal to
\[
\gaussian{n-j}{k-j}.
\]
\end{lemma}

We can now prove Proposition \ref{prop6}.

\begin{proof}[Proof of Proposition \ref{prop6}]
For the entire proof, we fix an
arbitrary positive integer $j>1$. As $ \Dav_1(C_2^r) = \Dav(C_2^r) = r + 1 $, we can ignore
the case $j=1$.

We shall now prove that for each integer $n$  larger than or equal
to $r+j$ (this condition is technically convenient later on) and less than
$\left( j \log 2 / \log (j+1) \right)\ r$ (notice that, for $r$ large
enough, since $j \geq 2$, such $n$'s always exist), one can find a
sequence of cardinality $n$ which does not contain $j$ disjoint
zero-sum subsequences. This will prove our result.

To each sequence $S=g_1,\dots, g_n$ over $C_2^r$, with $n \geq r+j$,
having the properties that $S$ does neither contain $0$ nor an element
at least twice, we associate, as described above, an $[n,n-r]$ code
(contained in $C_2^n$, and we fix some basis). This linear code,
automatically, has a minimal distance of at least $3$.
Conversely, any $[n,n-r,d]$ code with $d \geq 3$ can be obtained in this way (cf.~\cite{CZ2}).

The following remark is central: the condition that $S$ has $j$ disjoint
zero-sum subsequences translates to the condition that the associated code
contains $j$ non-zero codewords $c_1,\dots,c_j$ such that intersection
of the support (the set of indices of non-zero coordinates) of $c_u$ and
$c_v$ is empty for all distinct $u,v \in \{1, \dots, j\}$.
A code having this property will be called
{\em $j$-inadmissible}, otherwise it will be called {\em $j$-admissible}.

We first produce an upper bound on the total number of $[n,n-r]$ codes
that are $j$-inadmissible.

By definition any $j$-inadmissible code contains $c_1, \dots, c_j$
with the above mentioned property.
These $c_i$'s generate a $j$-dimensional vector space
since the $c_i$'s are certainly independent: the non-zero
coordinates of each $c_i$ are unique to that element.

Let $\mathcal{V}$ denote the set of all
subsets $\{d_1, \dots , d_j\} \subset C_2^n \setminus\{0\}$
such that the intersection of the support of $d_u$ and $d_v$
is empty for all distinct $u,v \in \{1, \dots ,j\}$; thus,
in particular, all the $d_i$'s are distinct.

We note that a code ${\mathcal C}$ is $j$-inadmissible if and only if
$V \subset {\mathcal C}$ for some $V \in \mathcal{V}$
(this $V$ is not  necessarily unique).
Moreover Lemma \ref{coeff} implies that for each $V \in \mathcal{V}$
there are $\gaussian{n-j}{n-r-j}$ codes containing
$V$; note that if $V \subset {\mathcal C}$ then ${\mathcal C}$
also contains the vector space generated by $V$, which is $j$-dimensional,
and apply Lemma \ref{coeff}. It follows that
the total number of $j$-inadmissible codes cannot exceed
\[
|\mathcal{V}|\gaussian{n-j}{n-r-j}.
\]

In order to estimate $|\mathcal{V}|$ we make the following remark:
each element of $\{ 1, \dots, n \}$ has to belong to either the support of
exactly one of the $d_i$'s or to none (obviously, the information
which element of $\{1, \dots , n\}$ belongs to each of the $d_i$'s
uniquely determines the element of $\mathcal{V}$). Thus,
for each element of $\{ 1, \dots, n \}$ there are (at most) $j+1$
possibilities.
This readily gives $|\mathcal{V}| \leq (j+1)^n$. (We ignore the
slight improvements that could be obtained from the fact that
the ordering of the $d_i$'s is irrelevant and the supports are non-empty,
as they would not affect our estimate).

We therefore infer that the total number of $j$-inadmissible $[n,n-r]$
codes is bounded above by
\[
(j+1)^n \gaussian{n-j}{n-r-j}.
\]

Again by Lemma \ref{coeff}, it follows that the ratio of the total
number of $j$-inadmissible $[n,n-r]$ codes divided by the total number
of $[n,n-r]$ codes is bounded above by
\begin{eqnarray*}
\frac{(j+1)^n \gaussian{n-j}{n-r-j}}{\gaussian{n}{n-r}} & = & (j+1)^n
\prod_{k=n-j+1}^n  \frac{2^{k-r}-1}{2^k -1}  \\
& \leq & (j+1)^n \prod_{k=n-j+1}^n \frac{2^{k-r}}{2^{k}}
\\
& = & (j+1)^n 2^{-rj}= 2^{ n \log_2 (j+1)-rj}.
\end{eqnarray*}
Here, $\log_2$ refers to the logarithm in basis $2$.

Thus, it follows that as soon as $(n \log_2 (j+1) - rj)$ is negative,
that is
\[
\frac{n}{r} < \frac{j}{\log_2 (j+1)},
\]
the existence of at least one admissible code is guaranteed.

From this we deduce (the condition $n \ge r + j$ becomes irrelevant)
\[
D_j (C_2^r) \geq \log 2 \frac{j}{\log (j+1)}\ r
\]
as $r$ tends to infinity.
\end{proof}

\section{Upper bounds}
\label{sec3}

\subsection{The crucial lemma}

The following lemma is central for our investigations.

\begin{lemma}
\label{centrallemma}
Let $f$ be an asymptotic upper-bounding function.
Let $j$ be a positive integer and $p$ be a real number
such that one has $\Dav_j(C_2^r)\le pr$ for each sufficiently large $r$.
Let finally $c$ denote a solution to the inequality
\[
\frac{p+c-1}{p+c} > f\left( \frac{c}{p+c} \right).
\]

Then for each sufficiently large integer $r$, we have
\[
\Dav_{j+1}(C_2^r)\le (p+c)r.
\]
\end{lemma}

\begin{proof}
By assumption, we have that
\[
\frac{(p+c)r - r}{(p+c)r} = \frac{p+c-1}{p+c} >
f\left( \frac{c}{p+c} \right)
=f\left( \frac{cr}{(p+c)r} \right) \geq f\left( \frac{cr + 1}{(p+c)r} \right),
\]
$f$ being decreasing. If we substitute $\lfloor (p+c)r \rfloor $
for $n$ and $\lfloor cr \rfloor$ for $d$,
we obtain (for $r$ sufficiently large and by the continuity of $f$)
\[
\frac{n-r}{n} > f \left( \frac{d+1}{n} \right),
\]
an equation of the type given in Lemma \ref{upperbounding} and
we may therefore deduce that
\[
\bound{\lfloor cr \rfloor}(C_2^r) \leq n = \lfloor (p+c)r \rfloor \leq (p+c)r.
\]
This now implies, by \eqref{eq_rec} and using our assumption, that
\begin{eqnarray*}
\Dav_{j+1} (C_2^r) & \le & \min_{i\in \mathbb{N}} \max \{ pr + i ,
\bound{i}(C_2^r) - 1 \} \\
        & \le & \max \{ pr + \lfloor cr \rfloor , \bound{\lfloor cr \rfloor }(C_2^r) - 1 \} \\
        & \le & \max \{ (p + c)r , (p+c)r - 1 \}= (p+c)r,
\end{eqnarray*}
as wanted.
\end{proof}

\subsection{The upper bounds in Theorem \ref{thm_expl}}

To obtain a proof of these upper bounds,
we use the approach described in Section \ref{sec2}
in combination with a bound on the parameters of linear codes originally due to
McEliece, Rodemich, Rumsey, and Welch \cite{MRRW},
that we recall here (see \emph{e.g.} \cite{MWS}):
Let us define $h$ to be the binary entropy function, that is
(for  $0 \leq u \leq 1$),
\[
h (u)=-u \log_2 u - (1-u) \log_2 (1-u)
\]
and $g(u)= h((1- \sqrt{1-u})/2)$. Then the function $f$ equal to
\begin{equation*}
\label{mrrw1}
f_1 (\delta) =
\begin{cases}
\min_{0\le u \le 1-2 \delta}
\left( 1 + g(u^2)-g(u^2 + 2\delta u +2\delta) \right) &
\text{ if $\delta \leq 1/2$} \\
0 & \text{ otherwise}
\end{cases}
\end{equation*}
is an asymptotically upper-bounding function and we may thus apply
Lemma \ref{centrallemma}.

We define a sequence $(u_j)_{j \in \N}$ recursively. We set $u_1=1$ and, for $j \geq 1$,
let $u_{j+1}$ be defined as the solution (if it exists, which is always
the case in practice, this solution has to be unique) to the equation
\[
1 - \frac{1}{U_j +u_{j+1}} =
f_1 \left( \frac{u_{j+1}}{U_j +u_{j+1}} \right)
\]
where we define $U_j = u_1 + \cdots + u_j$, the sum of the $j$ first
values of the sequence $(u_j)_{j \in \N}$. The sequence  $(U_j)_{j \in
  \N}$ corresponds to the coefficient
in the upper bound of Theorem \ref{thm_expl} ($U_1=1, U_2 =
1.395\dots, U_3 = 1.770\dots, \dots$), which therefore follows by
a repeated application of Lemma \ref{centrallemma}.

We point out that for our problem it is actually useful to use this
bound, as opposed to bounds whose numerical evaluation is simpler,
since for our problem we encounter $\delta$ in a fairly wide range.
A simpler strategy, regarding computations, which is --
in view of classical results on these bounds -- obviously worse,
though only slightly so, would be to use another bound proved in
\cite{MRRW} (see also \cite{MWS}), that is the function
\[
f_2 (\delta) =
\begin{cases}
h\left( 1/2  - \sqrt{\delta(1-\delta)} \right) & \text{ if $\delta \leq 1/2$} \\
0 & \text{ otherwise}
\end{cases}
\]
for the first few values, namely $2,3,4$ (for $2$ this yields
the identical bound, yet a slightly weaker one for $3,4$)
where $\delta$ is still fairly large, and then to switch
to using a bound that is better for small $\delta$ such
as the Elias--Bassalygo bound (see \cite{B} and \cite{MWS}), that is
\[
f_3 (\delta) =
\begin{cases}
1 - h \left( (1- \sqrt{1-2 \delta})/2 \right)   & \text{ if $\delta \leq 1/2$}\\
0 & \text{ otherwise}
\end{cases}
\]
with $h$ as above (this is the case $u=0$ of \eqref{mrrw1}).
Using this approach, we would for instance get
the values  $1.776$ for $j=3$, $2.147$ for $j=4$, $2.512$ for $j=5$
and $4.172$ for $j=10$ that is, slightly but noticeably weaker bounds.

\subsection{The upper bound in Theorem \ref{thm_asymp}}

For the asymptotics we use the method described in Section
\ref{sec2} in combination with the Hamming bound, that is
the function $f_4$ defined by
\begin{equation*}
\label{hamming}
f_4 ( \delta)= 1 - h \left( \frac{\delta}{2} \right).
\end{equation*}
We use this bound as the resulting analytic expressions and
asymptotic calculations are simpler than for stronger bounds.
On the other hand, asymptotically, using say Elias--Bassalygo
would not yield a better result.
Interestingly, for this particular choice of
upper-bounding function, a very elementary proof (cf.~below) of the
conclusion of Lemma \ref{centrallemma} can be obtained with additive
means, avoiding the coding theoretic argument that this function is
an upper-bounding function. This makes our proof of Theorem \ref{thm_asymp}
essentially self-contained.

\begin{proof}
We start with a sequence $S=g_1, \dots, g_n$ in $C_2^r$, where $n = \lfloor (p+c)r \rfloor$.
The set $\{1, \dots, n\}$ has
\[
\sum_{j=0}^{\lfloor cr/2 \rfloor} \binomial{\lfloor(p+c)r \rfloor}{j}
> \binomial{\lfloor(p+c)r\rfloor}{\lfloor cr/2\rfloor}
\]
subsets of size at most $cr/2$.
Since by assumption
\[
(p+c) \  h \left( \frac{c}{2(p+c)} \right) >1
\]
and since, as $r$ tends to infinity, $\binom{\alpha r}{\beta r} \sim c\ 2^{\alpha\ h (\beta /
  \alpha) r}$ (for some constant $c$ depending on $\alpha$ and $\beta$),
it follows that, for sufficiently large $r$, the number of such sets exceeds
$2^r$. This implies that there exist two distinct subsets ${\rm I, J}$
of $\{ 1,\dots, n \}$ such that
\[
\sum_{i \in {\rm I}} g_i = \sum_{i \in {\rm J}} g_i.
\]
This yields
\[
\sum_{i \in {\rm I} \triangle {\rm J}} g_i = 0,
\]
-- where ${\rm I} \triangle {\rm J}$ denotes the symmetric difference of
${\rm I}$ and ${\rm J}$ --, that is a  non-empty zero-sum sequence of
length at most $cr$.
\end{proof}

Using the upper-bounding function $f_4$,
we build a sequence, let us call it $(v_j)_{j \in \N}$
such that $v_1=1$ and
\begin{equation}
\label{seqV}
\frac{1}{V_j +v_{j+1}} = h \left( \frac{v_{j+1}}{2(V_j +v_{j+1})} \right),
\end{equation}
where $V_j = v_1 + \cdots + v_j$ for any integer $j \geq 1$.
It can be checked easily that such a sequence is well-defined and that for any
$j \geq 1$, one has $v_j \leq 1$. Then, rewriting \eqref{seqV} as
\[
1 - \frac{1}{V_j +v_{j+1}} = f_4 \left( \frac{v_{j+1}}{V_j +v_{j+1}} \right),
\]
we may repeatedly apply Lemma \ref{centrallemma} and obtain the inequality
\[
\Dav_{j}(C_2^r) \le V_j \ r
\]
for each integer $j$ and all sufficiently large $r$ (relative to $j$).

By the lower bounds in Theorem \ref{thm_asymp} proved in the preceding
Section, we already know that $V_j$ tends to infinity when $j$ tends
to infinity since
\begin{equation}
\label{lbr}
V_j \gtrsim \frac{D_j (C_2^r)}{r} \gg \frac{j}{\log j}
\end{equation}
as $r$ tends to infinity~; while $v_j$ remains bounded by $1$.

We can now develop \eqref{seqV} to obtain the desired asymptotics.
At the first order when $j$ tends to infinity we obtain
\[
\frac{1}{V_j+v_{j+1}} = - \frac{v_{j+1}}{2(V_j+v_{j+1})} \log_2
\left( \frac{v_{j+1}}{2 (V_j+v_{j+1})} \right)
  + O \left( \frac{v_{j+1}}{V_j+v_{j+1}} \right)
\]
or equivalently
\begin{eqnarray*}
\frac{2 \log 2}{v_{j+1}} & = & - \log \left( \frac{v_{j+1}}{2 (V_j+v_{j+1})} \right) +O(1) \\
             & = & \log ( V_j+v_{j+1} )  - \log v_{j+1} + O(1) \\
             & = & \log V_j  - \log v_{j+1} + O(1), \\
\end{eqnarray*}
and therefore
\[
\frac{2 \log 2}{v_{j+1}} + \log v_{j+1} = \log V_j + O(1).
\]
It follows that, as $j$ tends to infinity,
\[
v_{j+1} \sim \frac{2 \log 2}{\log V_j}.
\]

By \eqref{lbr}, we obtain
\[
V_{j+1} - V_{j}=v_{j+1} \lesssim \frac{2 \log 2}{\log j}
\]
and therefore, summing all these estimates yields
\[
V_j \lesssim 2 \log 2 \sum_{k=1}^{j-1} \frac{1}{\log k} \sim 2 \log 2
\frac{j}{\log j}
\]
from which the upper bound of Theorem \ref{thm_asymp} follows.

\section{Heuristics}
\label{sec_heu}

In this section, we discuss the quality of the bound in Theorem \ref{thm_expl}.
In investigations on intersecting codes, which are equivalent to determining $\Dav_2(C_2^r)$,
Cohen and Lempel \cite{CL} put forward the heuristic that one can expect
that the rate of an intersecting code will not exceed the Gilbert--Varshamov bound
\[1- h(\delta).\]
Extrapolating this heuristic to the investigation of $\Dav_j(C_2^r)$,
which might be too optimistic as the restrictions imposed on the code
corresponding to extremal example associated to $\Dav_j(C_2^r)$ get weaker as $j$ increases,
this suggests to use the function $1-h(\delta)$ as if it \emph{were} an upper-bounding function.

An argument similar to the one in the proof of Theorem \ref{thm_asymp}
thus yields an optimistic heuristic bound $\beta_j \lesssim \log 2
\frac{j}{\log j}$.

Regarding heuristic numerical values we get, for example, $1.294$ for $j=2$, $1.550$ for $j=3$,
$1.784$ for $j=4$, $2.003$ for $j=5$, and $2.984$ for $j=10$.

\section*{Acknowledgment}
The authors would like to thank G.~Cohen for helpful discussions.


\begin{thebibliography}{10}

\bibitem{Baa} P. C. Baayen, P. van Emde Boas, Een structuurconstante
  bepaald voor $C_2 \oplus C_2 \oplus C_2 \oplus C_2 \oplus C_6$,
 Mathematisch Centrum Amsterdam, WN 27 (1969), 6pp.

\bibitem{B} L. A. Bassalygo, New upper bounds for error-correcting codes,
Problemy Pereda{\v c}i Informacii 1(4) (1965), 41--44.

\bibitem{BS} G. Bhowmik, J.-Ch. Schlage-Puchta,
Davenport's constant for groups of the form
$\Bbb Z_3\oplus\Bbb Z_3\oplus\Bbb Z_{3d}$, pages 307--326, in:
Additive Combinatorics (Eds. A.~Granville, M.B.~Nathanson, and J.~Solymosi),
 CRM Proc. Lecture Notes, 43, Amer. Math. Soc., 2007.

\bibitem{CLZ} G. Cohen, S. Litsyn, G. Z\'emor,
Binary $B_2$-sequences: a new upper bound,
J. Combin. Theory Ser. A  94  (2001), 152--155.

\bibitem{CL} G. Cohen, A. Lempel, Linear intersecting codes,
Discrete Math. 56 (1985), 35--43.

\bibitem{CZ1} G. Cohen, G. Z\'emor, Intersecting codes and independent
  families, IEEE Trans. Inform. Theory 40 (1994), 1872--1881.

\bibitem{CZ2} G. Cohen, G. Z\'emor, Subset sums and coding theory,
Structure theory of set addition, pages 327--339, in: Structure Theory
of Set Addition (Eds. J.-M.~Deshouillers, B.~Landreau, and
A.A.~Yudin),  Ast{\'e}risque  258 (1999), Soci\'et\'e Math\'emathique
de France, 1999.

\bibitem{DOQ} Ch. Delorme, O. Ordaz, D. Quiroz,
Some remarks on Davenport constant, Discrete Math. 237 (2001), 119--128.

\bibitem{Edel} Y. Edel, Sequences in abelian groups $G$ of odd order
without zero-sum subsequences of length ${\rm exp}(G)$,
Des. Codes Cryptogr. 47 (2008), 125--134.

\bibitem{EEGKR} Y. Edel, Ch. Elsholtz, A. Geroldinger, S. Kubertin, L. Rackham,
Zero-sum problems in finite abelian groups and affine caps, Q. J. Math. 58 (2007),
 159--186.

\bibitem{vEB} P. van Emde Boas, A combinatorial problem
on finite abelian groups II, Mathematisch Centrum Amsterdam ZW 1969-007 (1969), 60pp.

\bibitem{vEBK} P. van Emde Boas, D. Kruyswijk, A combinatorial problem
on finite abelian groups III, Mathematisch Centrum Amsterdam ZW 1969-008 (1969), 34pp.

\bibitem{F} M. Freeze, Lengths of factorizations in Dedekind domains, Ph.D. thesis,
University of North Carolina at Chapel Hill, 1999.

\bibitem{FS} M. Freeze, W. A. Schmid, Remarks on a generalization of the Davenport
 constant, manuscript.

\bibitem{GG} W. D. Gao, A. Geroldinger, Zero-sum problems
in finite abelian groups: a survey, Expo. Math. 24 (2006), 337--369.

\bibitem{GB} A. Geroldinger, Additive group theory and non-unique factorizations,
pages 1--86,
in: Combinatorial Number Theory and Additive Group Theory (Eds. A. Geroldinger and
 I.Z. Ruzsa), Advanced Courses in Mathematics CRM Barcelona, Birkh{\"a}user, 2009.

\bibitem{GH} A. Geroldinger, F. Halter-Koch, Non-Unique Facorizations.
Algebraic, Combinatorial and Analytic Theory, Chapman \& Hall, 2006.

\bibitem{GS} A. Geroldinger, R. Schneider, On Davenport's constant,
J. Combin. Theory Ser. A  61 (1992), 147--152.

\bibitem{HK} F. Halter-Koch, A generalization of Davenport's constant and
its arithmetical applications, Colloq. Math. 63 (1992), 203--210.

\bibitem{KS} G. Katona, J. Srivastava, Minimal $2$-coverings of a finite
affine space based on ${\rm GF} (2)$, J. Statist. Plann. Inference 8 (1983),
375--388.

\bibitem{MWS} F. J. MacWilliams, N. J. A. Sloane,
The theory of error-correcting codes, North-Holland, 6th ed., 1988.

\bibitem{MRRW} R. McEliece, E. Rodemich, H. Rumsey, L. Welch, New
  upper bounds on the rate of a code via the Delsarte-MacWilliams
  inequalities, IEEE Trans. Information Theory 23 (1977), 157--166.

\bibitem{Meshu} R. Meshulam, An uncertainty inequality and zero subsums,
Discrete Math. 84 (1990), 197--200.

\bibitem{M} D. Mikl\'os, Linear binary codes with intersection
  properties, Discrete Appl. Math. 9 (1984), 187--196.

\bibitem{Ol} J. E. Olson, A combinatorial problem on finite Abelian groups I,
J. Number Theory 1 (1969), 8--10.

\bibitem{Ol2} J. E. Olson, A combinatorial problem on finite Abelian groups II,
J. Number Theory 1 (1969), 195--199.

\bibitem{Pot} A. Potechin, Maximal caps in $AG(6,3)$, Des. Codes Cryptogr., 46 (2008), 243--259.

\bibitem{Ro} K. Rogers, A combinatorial problem in Abelian groups,
Proc. Camb. Philos. Soc. 59 (1963), 559--562.

\end{thebibliography}
\end{document}